\documentclass[12pt]{article}
\usepackage{graphicx}
\usepackage{amsmath}
\usepackage{amsfonts}
\usepackage{amsthm}
\usepackage{color}
\usepackage[T1]{fontenc}
\usepackage{url}
\usepackage[margin=1in]{geometry}
\usepackage{amssymb,bm}
\usepackage{amsmath}
\usepackage{amsthm}
\usepackage{graphicx}
\usepackage[active]{srcltx} 
\usepackage{hyperref}
\hypersetup{pdfborder=0 0 0}

\newtheorem{theorem}{{\sc Theorem}}[section]

\newtheorem{lemma}[theorem]{{\sc Lemma}}
\newtheorem{corollary}[theorem]{Corollary}
\newtheorem{remark}[theorem]{Remark}

\newtheorem{conjecture}[theorem]{Conjecture}

\def\XXint#1#2#3{{\setbox0=\hbox{$#1{#2#3}{\int}$ }
\vcenter{\hbox{$#2#3$ }}\kern-.6\wd0}}

\bmdefine\BGa{\alpha}
\bmdefine\BGb{\beta}
\bmdefine\BGd{\delta}
\bmdefine\BGe{\epsilon}
\bmdefine\BGve{\varepsilon}
\bmdefine\BGf{\phi}
\bmdefine\BGvf{\varphi}
\bmdefine\BGg{\gamma}
\bmdefine\BGc{\chi}
\bmdefine\BGi{\iota}
\bmdefine\BGk{\kappa}
\bmdefine\BGl{\lambda}
\bmdefine\BGn{\eta}
\bmdefine\BGm{\mu}
\bmdefine\BGv{\nu}
\bmdefine\BGp{\pi}
\bmdefine\BGth{\theta}
\bmdefine\BGvth{\vartheta}
\bmdefine\BGr{\rho}
\bmdefine\BGvr{\varrho}
\bmdefine\BGs{\sigma}
\bmdefine\BGvs{\varsigma}
\bmdefine\BGt{\tau}
\bmdefine\BGj{\tau}
\bmdefine\BGu{\upsilon}
\bmdefine\BGo{\omega}
\bmdefine\BGx{\xi}
\bmdefine\BGy{\psi}
\bmdefine\BGz{\zeta}
\bmdefine\BGD{\Delta}
\bmdefine\BGF{\Phi}
\bmdefine\BGG{\Gamma}
\bmdefine\BGL{\Lambda}
\bmdefine\BGP{\Pi}
\bmdefine\BGT{\Theta}
\bmdefine\BGS{\Sigma}
\bmdefine\BGU{\Upsilon}
\bmdefine\BGO{\Omega}
\bmdefine\BGX{\Xi}
\bmdefine\BGY{\Psi}

\bmdefine\BCA{{\mathcal A}}
\bmdefine\BCB{{\mathcal B}}
\bmdefine\BCC{{\mathcal C}}
\bmdefine\BCD{{\mathcal D}}
\bmdefine\BCE{{\mathcal E}}
\bmdefine\BCF{{\mathcal F}}
\bmdefine\BCG{{\mathcal G}}
\bmdefine\BCH{{\mathcal H}}
\bmdefine\BCI{{\mathcal I}}
\bmdefine\BCJ{{\mathcal J}}
\bmdefine\BCK{{\mathcal K}}
\bmdefine\BCL{{\mathcal L}}
\bmdefine\BCM{{\mathcal M}}
\bmdefine\BCN{{\mathcal N}}
\bmdefine\BCO{{\mathcal O}}
\bmdefine\BCP{{\mathcal P}}
\bmdefine\BCQ{{\mathcal Q}}
\bmdefine\BCR{{\mathcal R}}
\bmdefine\BCS{{\mathcal S}}
\bmdefine\BCT{{\mathcal T}}
\bmdefine\BCU{{\mathcal U}}
\bmdefine\BCV{{\mathcal V}}
\bmdefine\BCW{{\mathcal W}}
\bmdefine\BCX{{\mathcal X}}
\bmdefine\BCY{{\mathcal Y}}
\bmdefine\BCZ{{\mathcal Z}}

\bmdefine\Bzr{ 0}
\bmdefine\Ba{ a}
\bmdefine\Bb{ b}
\bmdefine\Bc{ c}
\bmdefine\Bd{ d}
\bmdefine\Be{ e}
\bmdefine\Bf{ f}
\bmdefine\Bg{ g}
\bmdefine\Bh{ h}
\bmdefine\Bi{ i}
\bmdefine\Bj{ j}
\bmdefine\Bk{ k}
\bmdefine\Bl{ l}
\bmdefine\Bm{ m}
\bmdefine\Bn{ n}
\bmdefine\Bo{ o}
\bmdefine\Bp{ p}
\bmdefine\Bq{ q}
\bmdefine\Br{ r}
\bmdefine\Bs{ s}
\bmdefine\Bt{ t}
\bmdefine\Bu{ u}
\bmdefine\Bv{ v}
\bmdefine\Bw{ w}
\bmdefine\Bx{ x}
\bmdefine\By{ y}
\bmdefine\Bz{ z}
\bmdefine\BA{ A}
\bmdefine\BB{ B}
\bmdefine\BC{ C}
\bmdefine\BD{ D}
\bmdefine\BE{ E}
\bmdefine\BF{ F}
\bmdefine\BG{ G}
\bmdefine\BH{ H}
\bmdefine\BI{ I}
\bmdefine\BJ{ J}
\bmdefine\BK{ K}
\bmdefine\BL{ L}
\bmdefine\BM{ M}
\bmdefine\BN{ N}
\bmdefine\BO{ O}
\bmdefine\BP{ P}
\bmdefine\BQ{ Q}
\bmdefine\BR{ R}
\bmdefine\BS{ S}
\bmdefine\BT{ T}
\bmdefine\BU{ U}
\bmdefine\BV{ V}
\bmdefine\BW{ W}
\bmdefine\BX{ X}
\bmdefine\BY{ Y}
\bmdefine\BZ{ Z}

\begin{document}
\title{Quantitative anisotropic isoperimetric and Brunn-Minkowski inequalities for convex sets with improved defect estimates}
\author{ Davit Harutyunyan\footnote{EPFL, davit.harutyunyan@epfl.ch}}
\maketitle
\begin{abstract}
  In this paper we revisit the anisotropic isoperimetric and the Brunn-Minkowski inequalities for convex sets.
  The best known constant $C(n)=Cn^{7}$ depending on the space dimension $n$ in both inequalities
  is due to Segal [\ref{bib:Seg.}]. We improve that constant to $Cn^6$ for convex sets and to $Cn^5$ for centrally symmetric convex sets. We also conjecture, that the best constant in both inequalities must be of the form $Cn^2,$ i.e., quadratic in $n.$ The tools are the Brenier's mapping from the theory of mass transportation combined with new sharp geometric-arithmetic mean and some algebraic inequalities plus a trace estimate by Figalli, Maggi and Pratelli.
\end{abstract}



\section{Introduction and main results}
\label{sec:1}
The isoperimetric inequality is one of the classical inequalities in geometric measure theory, e.g., [\ref{bib:Fed.}]. It states the following:
\textit{If one prescribes the volume of a set in $\mathbb R^n,$ then its perimeter is smallest if and only if
the set is a ball. In the mathematical formulation one has the inequality
\begin{equation}
\label{1.1}
\mathrm{per}(E)\geq n|E|^{(n-1)/n}|B_1|^{1/n},
\end{equation}
for any measurable and bounded set $E\subset\mathbb R^n$ with perimeter $P(E)$ and Lebesgue measure $|E|.$ Moreover, the equality in (1.1) holds if and only if the set $E$ is a ball.} The isoperimetric inequality has been proven by different authors and different approaches, see the articles [\ref{bib:Mil.Sch.},\ref{bib:Kno.},\ref{bib:Fug.},\ref{bib:Fus.Mag.Pra.},\ref{bib:Fig.Mag.Pra.1},\ref{bib:Esp.Fus.Tro.}] and the references therein.
Upon introducing the isoperimetric deficit
\begin{equation}
\label{1.2}
\delta(E)=\frac{\mathrm{per}(E)}{n|E|^{(n-1)/n}|B_1|^{1/n}}-1\geq 0,
\end{equation}
of the set $E,$ one considers then the following stability question: \textit{Is it true, that if the deficit is close to zero, then the set $E$ is close to a ball in an appropriate sense?} A positive answer to this question has been given by many authors. In the case $n=2,$ Bonnesen [\ref{bib:Ben.}] (see also [\ref{bib:Vil.}]) proved that a planar domain $D$ of area $A$ that is bounded by a closed simple curve $\partial D$ of length $L,$ has concentric circles $O_1$ inside $D$ and $O_2$ containing $D$ and with radii $R_1$ and $R_2$ such, that
$$(R_2-R_1)^2\leq \frac{L^2}{4\pi}-A.$$
For the general case $n>2,$ the following version of the quantitative isoperimetric inequality has been considered:
\begin{equation}
\label{1.3}
\mathrm{per}(E)\geq n|E|^{(n-1)/n}|B_1|^{1/n}\left(1+\frac{(A(E))^\alpha}{C(n)}\right),
\end{equation}
where $A(E)$ is the asymmetry index of the set $E$ defined as
\begin{equation}
\label{1.4}
A(E)=\inf_{x\in\mathbb R^n}\left\{\frac{|E\triangle (x+B_r)|}{|B_r|} \ : \  |B_r|=|E|\right\},
\end{equation}
where $B_r$ is a ball with the same volume as $E,$ the set $X\triangle Y$ is the symmetric difference of the sets $X$ and $Y,$ and the sum $X+Y$ is the Minkowski sum defined as $X+Y=\{x+y \ : \ x\in X,\  y\in Y\}.$
Here the constant $C(n)$ depends on the space dimension $n$ and $\alpha\in\mathbb R$ is a positive number. There is also the so called anisotropic or weighted version of the isoperimetric inequality which we present below. Assume $L\subset \mathbb R^n$ is an open bounded convex set that contains the origin. Define the weight function of the set $L$ in all directions in $\mathbb R^n$ as follows:
\begin{equation}
\label{1.4.1}
\|\nu\|_{\ast}=\sup\{x\cdot\nu \ : \ x\in L\},\quad\text{for all directions}\quad \nu\in\mathbb S^{n-1}.
\end{equation}
Let now $E\in\mathbb R^n$ be a piecewise smooth\footnote{A set that has a piecewise smooth boundary} open set oriented by the outer unit normal $\nu_E.$ Then the anisotropic or weighted perimeter of $E$ with respect to $L$ is defines to be
\begin{equation}
\label{1.4.2}
P_L(E)=\int_{\partial E}\|\nu_E\|_{\ast}d\mathcal{H}^{n-1}.
\end{equation}
In the case when $L$ is the unit ball centered at the origin, $P_L(E)$ coincides with the usual perimeter $\mathrm{per}(E)$ of the set $E.$
The anisotropic isoperimetric inequality then states, that if the volume of $E$ is fixed, then the anisotropic perimeter $P_L(E)$ is minimised for the set $E$ that is homothetic to $L,$ which is Wulff's conjecture [\ref{bib:Wul.}], see the work of Fonseca and M\"uller [\ref{bib:Fon.Mue.}] for a proof. The Wulff inequality reads as
$$P_L(E)\geq n|E|^{(n-1)/n}|L|^{1/n},$$
for all $E\in\mathbb R^n$ open bounded domains. In what follows, we assume that $L$ is a fixed convex set and we will
drop the dependence in $L$ of some parameters sometimes in order to not to complicate the notation. The anisotropic deficit will then be the quantity
\begin{equation}
\label{1.4.3}
\delta(E)=\frac{P_L(E)}{n|E|^{(n-1)/n}|L|^{1/n}}-1\geq 0.
\end{equation}
The quantitative version of (\ref{1.4.3}) analogous to (\ref{1.4}) then naturally arises and reads as
\begin{equation}
\label{1.4.4}
\delta(E)\geq \frac{(A(E))^\alpha}{C(n)}.
\end{equation}
The quantity $A(E)$ here is not the asymmetry index of the set $E,$ but rather it determines the amount of how much the shape of $E$ differs from the shape of $L,$ i.e.,
\begin{equation}
\label{1.4.5}
A(E)=\inf_{x\in\mathbb R^n}\left\{\frac{|E\triangle (x+rL)|}{|E|} \ : \  |rL|=|E|\right\}.
\end{equation}
\textit{The main question is then what the smallest possible value of $\alpha$ and $C(n)$ are.}
Inequality (\ref{1.3}) dates back to 1905, when Bernstein [\ref{bib:Ber.}] studied it in the case $n=2.$
Quantitative versions of the isoperimetric inequality (\ref{1.3}) have been proven again by different authors, among which the first
proof for arbitrary Borel sets is due to Fusco, Maggi and Pratelli [\ref{bib:Fus.Mag.Pra.}], where the authors prove that (\ref{1.3}) holds
for $\alpha=2$ and some constant $C(n)$ proving Hall's conjecture [\ref{bib:Hal.}].
The work of Figalli, Maggi and Pratelli [\ref{bib:Fig.Mag.Pra.1}] by a mass transportation approach then follows, where the authors pursue Gromov's approach [\ref{bib:Mil.Sch.}] to prove the inequality (\ref{1.4.4}), i.e., the anisotropic case (Wulff's inequalty) for $\alpha=2$ and an improved constant $C(n).$ The mass transportation approach has been know to be an excellent tool for proving geometric inequalities, e.g., [\ref{bib:McC.1.},\ref{bib:McC2.},\ref{bib:Bal.}]. As already said, the optimal value of $\alpha$ in the inequalities (\ref{1.3}) and (\ref{1.4.4}) is $\alpha=2.$ Until the year 2010, the best know constant $C(n)$ in (\ref{1.3}) and (\ref{1.4.4}) was obtained by Figalli, Maggi and Pratelli in [\ref{bib:Fig.Mag.Pra.1}], where they get the value $C(n)=C(n)=\frac{181n^7}{(2-2^{(n-1)/n})^{1.5}}.$
It is easy to see, that $Cn^{8.5}\geq C(n)\geq cn^{8.5}$ for all $n\in\mathbb N,$ i.e., it has a polynomial growth. In the sequel we will refer to
both inequalities (\ref{1.3}) and (\ref{1.4.4}) for the value $\alpha=2.$ Then the work of Segal [\ref{bib:Seg.}] followed in 2012, where following the lines of [\ref{bib:Fig.Mag.Pra.1}], Segal improved the constant to $Cn^7,$ which is the best known constant to our best knowledge.
In this paper we aim to prove the inequality (\ref{1.4.4}) with an improved constant $C(n)=100n^6$
for any convex body $E$ and with $C(n)=100n^5$ for any convex body $E$ that is centrally symmetric.
Our strategy is to prove a somewhat specialized Brunn-Minkowski inequality and then to derive the isoperimetric inequality from it.
Let us introduce the Brunn-Minkowski inequality. In the beginning of the 20th century Minkowski proved that for any measurable bounded sets $X,Y\in\mathbb R^n,$ the inequality holds:
\begin{equation}
\label{1.5}
|X+Y|^{1/n}\geq |X|^{1/n}+|Y|^{1/n}.
\end{equation}
 Inequality (\ref{1.5}) is called Brunn-Minkowsky inequality for sets.
The credit of Brunn in (\ref{1.5}) is that he had proved it for the case $n=3$ before Minkowsky's general proof.
It has been proved [\ref{bib:Hen.Mac.},\ref{bib:Bur.Zal.},\ref{bib:Gro.1},\ref{bib:Gro.2}], that equality holds in (\ref{1.5}) if and only if the sets $X$ and $Y$ are homothetic to the same convex set, i.e., there exists a convex set $K\in\mathbb R^n,$ two vectors $u,v\in\mathbb R^n$ and numbers $\lambda,\mu>0$ such, that
$|(u+\lambda K)\triangle X|=|(v+\mu K)\triangle Y|=0.$ An analogous quantitative version of (\ref{1.5}) is as follows, e.g., [\ref{bib:Fig.Mag.Pra.2}],
\begin{equation}
\label{1.6}
|X+Y|^{1/n}\geq (|X|^{1/n}+|Y|^{1/n})\left(1+\frac{A(X,Y)^2}{C_0(n)\sigma(X,Y)^{1/n}}\right),
\end{equation}
for all bounded convex sets $X,Y\subset\mathbb R^n,$ where
\begin{align}
\label{1.7}
A(X,Y)&=\inf_{x\in\mathbb R^n}\left\{\frac{|X\triangle (x+\lambda Y)|}{|X|}\ : \ \lambda=\left(\frac{|X|}{|Y|}\right)^{1/n}\right\},\\ \nonumber
\sigma(X,Y)&=\max\left(\left(\frac{|X|}{|Y|}\right)^{1/n},\left(\frac{|Y|}{|X|}\right)^{1/n}\right).
\end{align}
As it is know, e.g., [\ref{bib:Sch.}], that the Brunn-Minkowski inequality (\ref{1.6}) implies the stable version of the isoperimetric inequality (\ref{1.4.4}) with the same constant $C(n)=C_0(n).$ In the existing works on the Brunn-Minkowski inequality, the dependence of the constant $C_0(n)$ upon the space dimension $n$ is $Cn^7,$ which is due to Segal [\ref{bib:Seg.}], thus it is our task to derive a version of (\ref{1.6}) with a constant $C(n)$ that depends on $n$ relatively favorably. We believe, that the asymptotically best constant $C$ in all inequalities (\ref{1.3}), (\ref{1.4.4}) and (\ref{1.6}) is of the form $Cn^2$ as $n\to\infty,$ which we conjecture in this paper.
Recall, that in the general case when the sets $X$ and $Y$ are not convex, the stability of the classical Brunn-Minkowski inequality has been proven by Figalli and Jerison [\ref{bib:Fig.Jer.}], where the authors prove a version of (\ref{1.6}) with $A(X,Y)^{\alpha_n}$ instead of  $A(X,Y)^{2}$ with some $\alpha_n>0$ depending on $n$ and having exponential growth. Also, if one of the sets $X$ and $Y$ is convex, then Carlen and Maggi [\ref{bib:Car.Mag.}] proved an estimate analogous to (\ref{1.6}) with $A(X,Y)^{4}$ instead of  $A(X,Y)^{2}$ with some constant $C_0(n).$ In general, even in the case when only one of the sets $X$ and $Y$ is convex, the best exponent $\alpha$ of $A(X,Y)$ (which is 2 for convex sets) is not know. Let us now introduce some more notation for convex sets. In what follows, we will use the letters $K$ and $L$ for convex sets to keep the notation consistent with the monograph on convex bodies and the Brunn-Minkowski theory by Schneider in [\ref{bib:Sch.}].
Given a bounded domain $\Omega\in\mathbb R^n,$ denote by $r_\Omega$ and $R_\Omega$ the inner and the outer radii of it, i.e., $r_\Omega$ is the radius of the biggest ball that can be put in $\Omega$ and $R_\Omega$ is the radius of the smallest ball that contains $\Omega.$ It is well known that any compact convex body has a minimal ball inside itself and a maximal ball containing it, e.g., [\ref{bib:Sch.}].
Given now a compact convex body, denote the so called inverse roundness of $K$ to be the quantity
\begin{equation}
\label{1.8}
q_{K}=\inf\left\{\frac{R_{T(K)}}{r_{T(K)}} \ : \ T=Ax, \  A\in M^{n\times n},\ \  \det{A}\neq 0\right\}.
\end{equation}
The quantity $q_K$ determines how round the convex body $K$ can be made by a nonsingular affine transformation: the smaller the value
of $q_K$ is the rounder the body $K$ can be made.
Another property of convex bodies is that any convex body $K\in\mathbb R^n$ fulfills the inequality, e.g. [\ref{bib:Joh.},\ref{bib:Sch.}],
\begin{equation}
\label{1.9}
\frac{R_{T(K)}}{r_{T(K)}}\leq n,
\end{equation}
where the transformation $T$ is called John's symmetrization [\ref{bib:Joh.}]. It is also known, that if the convex body $K$ is centrally symmetric, then one has an improved version of (\ref{1.9}), [\ref{bib:Hor.},\ref{bib:Sch.}], namely,
\begin{equation}
\label{1.10}
\frac{R_{T(K)}}{r_{T(K)}}\leq \sqrt n.
\end{equation}

We hereafter assume that $n\geq 2.$ Next come the main results of the paper.
\begin{theorem}
\label{th:1.1}
Assume $K,L\subset\mathbb R^n$ are compact convex bodies such that $|K|,|L|>0.$ Then the quantitative anisotropic isoperimetric inequality holds:
\begin{equation}
\label{1.11}
P_L(K)\geq n|K|^{(n-1)/n}|L|^{1/n}\left(1+\frac{(A(K))^2}{C(K,n)}\right),
\end{equation}
with the constant $C(K,n)=100n^4q_K^2.$
\end{theorem}

\begin{corollary}
\label{cor:1.2}
Owing to the estimates (\ref{1.9}) and (\ref{1.10}) we get the that the quantitative anisotropic isoperimetric inequality (\ref{1.11}) holds with
the constant $C(n)=100n^6$ for any convex bodies $K$ and with the constant $C(n)=100n^5$ provided the convex body $K$ is centrally symmetric.
\end{corollary}
The following version of Brunn-Minkowski inequality then follows:
\begin{theorem}
\label{th:1.3}
Assume $K,L\subset\mathbb R^n$ are compact convex bodies such that $|K|,|L|>0.$ Then the quantitative Brunn-Minkowski inequality holds:
\begin{equation}
\label{1.12}
|K+L|^{1/n}\geq (|K|^{1/n}+|L|^{1/n})\left(1+\frac{(A(K,L))^2}{C(n)\sigma(K,L)^{1/n}}\right),
\end{equation}
with the constant $C(n)=400n^6.$ Recall, that inequality (\ref{1.12}) reads as
$$\beta(K,L)\geq \frac{(A(K,L))^2}{C(n)\sigma(K,L)^{1/n}},$$
where
\begin{equation}
\label{1.12.1}
\beta(K,L)=\frac{|K+L|^{1/n}}{|K|^{1/n}+|L|^{1/n}}-1
\end{equation}
is the Brunn-Minkowski deficit, e.g., [\ref{bib:Fig.Mag.Pra.1}].
\end{theorem}
We also make the following conjecture:
\begin{conjecture}
\label{con:1.4}
In both Theorems~\ref{th:1.1},\ref{th:1.3}  the optimal constants $C(K,n)$ and $C(n)$ are of the form $Cn^2,$ where $C$ is an absolute constant.
 \end{conjecture}

\section{Stable geometric-arithmetic mean inequalities and the connection with the Brunn-Minkowski inequality}
\label{sec:2}
\setcounter{equation}{0}
In this section we prove sharp quantitative versions of the geometric and arithmetic mean inequality. The purpose of that
is then to use them in the derivation of a stable Brunn-Minkowski inequality. Our motivation is as follows: It is very
well known, that Brunn-Minkowski inequality can be derived from the arithmetic-geometric mean inequality on a page as done by
Hadwiger and Ohmann in [\ref{bib:Had.Ohm.}], see also [\ref{bib:Har.Lit.Pol.},\ref{bib:Sch.}] and
the celebrated review article of Gardiner [\ref{bib:Gar.}] for details.
The interesting thing is that the reverse process can also be done in a few lines, i.e,
the geometric-arithmetic mean inequality can be derived from the Brunn-Minkowski inequality.
To our best knowledge that has never been written anywhere and we present it in the below lemma.
\begin{lemma}
\label{lem:2.1}
The Brunn-Minkowski inequality implies the geometric and arithmetic mean inequality,
\begin{equation}
\label{2.1}
\frac{x_1+x_2+\dots+x_n}{n}\geq (x_1x_2\dots x_n)^{1/n},\quad\text{for all} \quad x_i\geq 0, \ \ i=1,2,\dots,n.
\end{equation}
i.e., they are equivalent.
\end{lemma}

\begin{proof}
Assume (\ref{1.5}) is satisfied for the sets $X=K$ and $Y=L$ according to our convention. Then take the sets $K=[0,\epsilon]^n$ and
$L=[0,x_1]\times[0,x_2]\times\dots\times[0,x_n]$ and apply (\ref{1.5}) to the pair $(K,L)$ to get
$$\epsilon+(x_1x_2\dots x_n)^{1/n}\leq ((x_1+\epsilon)(x_2+\epsilon)\dots (x_n+\epsilon))^{1/n}.$$
Taking now the $n-th$ power of both sides we get after the cancellation of $x_1x_2\dots x_n,$
$$\epsilon n(x_1x_2\dots x_n)^{1/n}+O(\epsilon^2)\leq \epsilon (x_1+x_2+\dots+x_n)+O(\epsilon^2).$$
Sending now $\epsilon$ to zero, we arrive at (\ref{2.1}). The proof is finished now.
\end{proof}
This equivalence suggests, that one may be able to prove a quantitative version of the Brunn-Minkowski inequality via a
quantitative version of the geometric-arithmetic mean inequality. We prove the following theorem.
\begin{theorem}
\label{th:2.2}
For a sequence $x_1,x_2,\dots,x_n\geq 0$ denote $x=(x_1x_2\dots x_n)^{1/n}.$ Then the following quantitative version of the geometric-arithmetic mean inequality holds:
\begin{equation}
\label{2.2}
\frac{x_1+x_2+\dots+x_n}{n}\geq (x_1x_2\dots x_n)^{1/n}+\frac{1}{n}\sum_{i=1}^n(\sqrt{x_i}-\sqrt{x})^2.
\end{equation}
Moreover the equality holds if and only if one of the numbers $x_1,x_2,\dots,x_n$ is zero or if all of them are equal.
\end{theorem}

\begin{proof}
The proof is trivial, we simply open the brackets on the right to get an equivalent inequality
$$\frac{x_1+x_2+\dots+x_n}{n}\geq x+\frac{x_1+x_2+\dots+x_n}{n}+x-\frac{2\sqrt{x}}{n}\sum_{i=1}^n\sqrt{x_i},$$
which is equivalent to
$$\frac{1}{n}\sum_{i=1}^n\sqrt{x_i}\geq \sqrt{x},$$
i.e., the geometric and arithmetic mean inequality for the sequence $\sqrt{x_1}, \sqrt{x_2},\dots,\sqrt{x_n}.$
If $x>0$ then it is clear that the equality holds if and only if $x_i=x$ for all $1\leq i\leq n.$ If $x=0,$ then clearly equality holds in (\ref{2.2}). If $x=0,$ then again clearly equality holds in (\ref{2.2}).
\end{proof}

\begin{remark}
\label{rem:2.3}
The coefficient $\frac{1}{n}$ in front of the expression $\sum_{i=1}^n(\sqrt{x_i}-\sqrt{x})^2$ on the right can not be improved
as shown by the example $x_1=1, x_i=0, i\geq 2.$
\end{remark}

\begin{corollary}
\label{cor:2.4}
For any sequence $x_1,x_2,\dots,x_n\geq 0$ the inequality holds:
\begin{equation}
\label{2.3}
\frac{x_1+x_2+\dots+x_n}{n}\geq (x_1x_2\dots x_n)^{1/n}+\frac{1}{2n}\sum_{i=1}^n\frac{(x_i-x)^2}{x_i+x}.
\end{equation}
\end{corollary}

\begin{proof}
The proof is a direct consequence of inequality (\ref{2.2}) and the estimate
$$(\sqrt{x_i}-\sqrt{x})^2=\frac{(x_i-x)^2}{(\sqrt{x_i}+\sqrt{x})^2}\geq \frac{(x_i-x)^2}{2(x_i+x)}.$$
\end{proof}

The next theorem is an alternative version of Theorem~\ref{th:2.2} which may be of separate interest.
\begin{theorem}
\label{th:2.5}
Assume $x_1,x_2,\dots,x_n\geq 0.$ Then the following quantitative version of the geometric-arithmetic mean inequality holds:
\begin{equation}
\label{2.4}
\frac{x_1+x_2+\dots+x_n}{n}\geq (x_1x_2\dots x_n)^{1/n}+\frac{1}{n(n-1)}\sum_{1\leq i<j\leq n}(\sqrt{x_i}-\sqrt{x_j})^2.
\end{equation}
Moreover the equality holds only in one of the following cases:
\begin{itemize}
\item[(i)] If $n=2$.
\item[(ii)] All but one of the numbers $x_1,x_2,\dots,x_n$ are zero.
\item[(iii)] All of the numbers $x_1,x_2,\dots,x_n$ are equal.
\end{itemize}

\end{theorem}

\begin{proof}
The proof is again trivial, upon opening the brackets on the right we get an equivalent inequality
\begin{equation}
\label{2.5}
\frac{2}{n(n-1)}\sum_{1\leq i<j\leq n}\sqrt{x_ix_j}\geq (x_1x_2\dots x_n)^{1/n},
\end{equation}
which is exactly the geometric-arithmetic mean inequality for the numbers $\sqrt{x_ix_j}.$
It is also clear, that the equality in (\ref{2.5}) will hold if and only if all the numbers $\sqrt{x_ix_j}$ are equal. It is clear that the case $n=2$ provides equality in (\ref{2.4}). Assume now $n\geq 3.$ If $x_i\neq 0$ for some $1\leq i\leq n,$ then we get from the equality $\sqrt{x_ix_j}=\sqrt{x_ix_k}$ that $x_j=x_k$ for $j,k\neq i.$ On the other hand as $n\geq 3$ the equality $\sqrt{x_ix_j}=\sqrt{x_jx_k}$ holds and thus we get $x_j(x_i-x_k)=0$ for all $j,k\neq i$ and $j\neq k.$ This then implies that $x_j=0$ for all $j\neq i$ or $x_i=x_j$ for all $1\leq i,j\leq n,$ which are exactly cases $(ii)$ and $(iii)$ respectively. It is trivial that both cases provide equality in (\ref{2.5}) The proof is finished now.
\end{proof}

\begin{remark}
\label{rem:2.6}
The constant $\frac{1}{n(n-1)}$ can not be improved in the inequality (\ref{2.4}) as shown by the example $x_1=1, x_i=0, i\geq 2.$
\end{remark}
The last theorem in this section provides another key estimate in the proof of Theorem~\ref{th:1.1}.

\begin{remark}
\label{rem:2.6.1}
Unlike the classical geometric-arithmetic mean inequality, there are several equality cases in both (\ref{2.2}) and (\ref{2.4}).
\end{remark}

\begin{lemma}
\label{lem:2.7}
For any numbers $x_1,x_2,\dots,x_n\in\left[0,\frac{1}{2}\right]$ the inequality holds:
\begin{equation}
\label{2.6}
\sum_{i=1}^n\sqrt{\frac{x_i}{1+x_i}}\geq n\sqrt{\frac{x}{1+x}},
\end{equation}
where as before we define $x=(x_1x_2\dots x_n)^{1/n}.$
\end{lemma}

\begin{proof}
We prove the theorem by induction up and down in $n$, namely the strategy is to prove inequality (\ref{2.6}) for $n=2^k$ and then
derive it for $n-1$ provided it holds for $n.$ The first step is to prove (\ref{2.6}) for $n=2,$ i.e.,
\begin{equation}
\label{2.7}
\sqrt{\frac{x_1}{1+x_1}}+\sqrt{\frac{x_2}{1+x_2}}\geq 2\sqrt{\frac{x}{1+x}},
\end{equation}
where $x=\sqrt{x_1x_2}.$ Denote $x_i=\tan^2(\alpha_i),$ where $\alpha_i\in[0,\alpha]\subset[0,\frac{\pi}{2})$ such that $\tan\alpha\leq\frac{1}{\sqrt 2}.$
After some trigonometric manipulation, inequality (\ref{2.6}) turns into the form
$$(\sin\alpha_1+\sin\alpha_2)^2\cos(\alpha_1-\alpha_2)\geq 4\sin\alpha_1\sin\alpha_2,$$
which is the same as
$$(\sin\alpha_1+\sin\alpha_2)^2\left[1-2\sin^2\frac{\alpha_1-\alpha_2}{2}\right]\geq 4\sin\alpha_1\sin\alpha_2,$$
which is equivalent to
$$(\sin\alpha_1-\sin\alpha_2)^2\geq 2\sin^2\frac{\alpha_1-\alpha_2}{2}(\sin\alpha_1+\sin\alpha_2)^2,$$
and again after some trigonometry we arrive at an equivalent form
$$4\sin^2\frac{\alpha_1-\alpha_2}{2}\cos^2\frac{\alpha_1+\alpha_2}{2}\geq 2\sin^2\frac{\alpha_1-\alpha_2}{2}(\sin\alpha_1+\sin\alpha_2)^2,$$
i.e., the form
\begin{equation}
\label{2.8}
2\cos^2\frac{\alpha_1+\alpha_2}{2}\geq (\sin\alpha_1+\sin\alpha_2)^2.
\end{equation}
It is clear, that
$$2\cos^2\frac{\alpha_1+\alpha_2}{2}\geq 2\cos^2\alpha$$
 and
 $$(\sin\alpha_1+\sin\alpha_2)^2\leq 4\sin^2\alpha,$$
  and thus (\ref{2.8})
follows from the fact $\tan\alpha\leq\frac{1}{\sqrt 2},$ thus (\ref{2.7}) is proven. If now inequality (\ref{2.6}) is true for $n=k,$ then we have for $2k$ numbers $x_1,x_2,\dots,x_k,x_{k+1},\dots,x_{2k}\in[0,\frac{1}{2}],$ that
\begin{align*}
\sum_{i=1}^{2k}\sqrt{\frac{x_i}{1+x_i}}&=\sum_{i=1}^{k}\sqrt{\frac{x_i}{1+x_i}}+\sum_{i=k+1}^{2k}\sqrt{\frac{x_i}{1+x_i}}\\
&\geq k\left(\sqrt{\frac{(\prod_{i=1}^kx_i)^{1/k}}{1+(\prod_{i=1}^kx_i)^{1/k}}}+\sqrt{\frac{(\prod_{i=k+1}^{2k}x_i)^{1/k}}{1+(\prod_{i=k+1}^{2k}x_i)^{1/k}}}\right)\\
&\geq 2k\sqrt{\frac{x}{1+x}},
\end{align*}
hence, (\ref{2.6}) is true for $2k$ numbers too. By induction (\ref{2.6}) holds true for any number $n=2^k,$ $k\in\mathbb N.$ Observe, now, that is (\ref{2.6}) holds for $n+1$ numbers, then given the sequence $x_1,x_2,\dots,x_n\in[0,\frac{1}{2}],$ we can utilizing it for the $n+1$ numbers
$x_1,x_2,\dots,x_n,x\in[0,\frac{1}{2}],$ to derive (\ref{2.6}) exactly for the sequence $x_1,x_2,\dots,x_n.$ The proof of the lemma is finished now.

\end{proof}

\section{Proof of the main result}
\label{sec:3}
\setcounter{equation}{0}
We adopt the mass transportation approach proposed by Gromov e.g.,  [\ref{bib:Mil.Sch.}] and successfully employed by
Ball in [\ref{bib:Bal.}] and Figalli, Maggi and Pratelli in [\ref{bib:Fig.Mag.Pra.2}].
In fact our proof is a refinement of the Figalli, Maggi, Pratelli proof in [\ref{bib:Fig.Mag.Pra.2}] where the estimates (\ref{2.2}) and (\ref{2.6}) and their careful application play a significant role.

\begin{proof}[Proof of Theorem~\ref{th:1.1}] First we prove the following suitably modified partial version of the Brunn-Minkowski inequality.

\begin{lemma}
\label{lem:3.1.1}
Assume $K,L\subset\mathbb R^n$ are convex bodies such that $|K|,|L|>0$ and denote $L_\epsilon=\epsilon L$ for any $\epsilon\in\mathbb R.$
Then there exists a positive constant $\epsilon_0=\epsilon_0(K,L)>0$ that depends on the convex sets $K$ and $L$ such, that for any $\epsilon\in (0,\epsilon_0),$ the inequality holds:
\begin{equation}
\label{3.0.1}
|K+L_\epsilon|^{1/n}\geq (|K|^{1/n}+|L_\epsilon|^{1/n})\left(1+\frac{A(K,L_\epsilon)^2}{C(K,n)\sigma(K,L_\epsilon)^{1/n}}\right),
\end{equation}
with the constant $C(K,n)=100n^4q_K^2.$
\end{lemma}
\begin{remark}
\label{rem:3.1.2}
The important point is that the constant $C(K,n)$ in the inequality (\ref{3.0.1}) does not depend on the set $L.$ This fact will be crucial
when deriving an isoperimetric inequality from Lemma~\ref{lem:3.1.1}.
\end{remark}

\begin{proof}
It is clear, that if the constant $C(K,n)$ is frozen, then the inequality (\ref{3.0.1}) is affine transformation-invariant, thus as the constant
$C(K,n)$ does not depend on the set $L,$ then we can without loss of generality assume, that the set $K$ has maximal possible roundedness, i.e.,
$K$ is such that the quantity $q_K$ in (\ref{1.8}) is achieved for the identical transformation $T(x)=x.$
An inner approximation, i.e., an approximation of the sets $K$ and $L$ from inside by other compact convex sets
lowers the measure of the sum $|K+L_\epsilon|$ and approximates the quantities  $|K|,$ $|L_\epsilon|,$ $A(K,L_\epsilon)$ and $\sigma(K,L_\epsilon)$, thus we can assume without loss of generality, that the sets $K$ and $L$ are smooth and uniformly convex. Due to the convexity of the sets $K$ and $L$, one has by Brenier's theorem, [\ref{bib:Bre.1},\ref{bib:Bre.2},\ref{bib:Vil.}] that there exists a convex function $\varphi(x)\colon\mathbb R^n\to\mathbb R$ such that the its gradient $F=\nabla\varphi\colon\mathbb R^n\to\mathbb R^n$ is a function of bounded variation, $F\in BV(\mathbb R^n, L)$ and pushes forward the probability measure $\frac{1}{|K|}\chi_{K}dx$ to the probability measure $\frac{1}{|L|}\chi_{L}dx,$ i.e., it has a constant Jacobian in $K$:
\begin{equation}
\label{3.1}
\det F(x)=\frac{|L|}{|K|},\quad\text{for all}\quad x\in K.
\end{equation}
Next, we can without loss of generality assume, that $|K|>|L|,$ thus we get $\sigma(K,L)=\frac{|K|}{|L|}$. As long as the sets $K$ and $L$ are smooth and uniformly convex, Caffarelli showed in [\ref{bib:Caf.1},\ref{bib:Caf.2}], that Brenier's map is smooth up to the boundary of $K,$ i.e., $F\in C^\infty(\overline{K},\overline{L})$. On the other hand by the convexity of the map $\varphi,$ the Hessian
$\nabla^2\varphi(x)=\nabla F(x)$ is a symmetric positive semi-definite matrix for all $x\in K.$ Denoting the eigenvalues of $\nabla F$ by
$\lambda_1(x),\lambda_2(x),\dots,\lambda_n(x),$ they must be real and positive, and we get according to the condition (\ref{3.1}) that,
\begin{equation}
\label{3.2}
\prod_{i=1}^n\lambda_i(x)=\frac{|L|}{|K|}=\mu^n,\quad\text{for all}\quad x\in K.
\end{equation}
where $\mu=\left(\frac{|L|}{|K|}\right)^{1/n}=\frac{1}{\sigma(K,L)^{1/n}}.$ It is then clear that we can use the map $\varphi_\epsilon(x)=\epsilon\varphi(x)\colon\mathbb R^n\to\mathbb R$ as Brenier's map for $\epsilon>0$ and for the sets $K$ and $L_\epsilon.$
Denote next $F_\epsilon(x)=\epsilon F(x)$ and $G_\epsilon(x)=x+F_\epsilon(x).$ It is clear, that $G_\epsilon\colon K\to K+L_\epsilon,$ thus $G_\epsilon(K)\subset K+L_\epsilon.$ Let us now verify, that the map $G_\epsilon\colon K\to K+L_\epsilon$ is surjective. Assume in contradiction, that
$G_\epsilon(x_1)=G_\epsilon(x_2),$ for some $x_1,x_2\in K,$ with $x_1\neq x_2$. Thus we get
\begin{equation}
\label{3.3}
F_\epsilon(x_1)-F_\epsilon(x_2)=x_2-x_1.
\end{equation}
By the mean value formula we have $F_\epsilon(x_1)-F_\epsilon(x_2)=\nabla F_\epsilon(\theta x_1+(1-\theta)x_2)(x_1-x_2)$ for some $\theta\in [0,1],$ thus owing to (\ref{3.3})
we obtain $(\nabla F_\epsilon(\theta x_1+(1-\theta)x_2)+I)(x_1-x_2)=0$ which gives
\begin{equation}
\label{3.4}
\det[\nabla F_\epsilon(\theta x_1+(1-\theta)x_2)+I]=0.
\end{equation}
Recall now, that the Hessian $\nabla F_\epsilon(\theta x_1+(1-\theta)x_2)$ is positive definite and thus so is the sum
$\nabla F_\epsilon(\theta x_1+(1-\theta)x_2)+I$ which contradicts (\ref{3.4}). From the surjectivity of the map $G_\epsilon\colon K\to K+L_\epsilon$ and the fact $G_\epsilon(K)\subset K+L_\epsilon,$ we obtain
\begin{equation}
\label{3.5}
|K+L_\epsilon|\geq |G_\epsilon(K)|=\int_K \det \nabla G_\epsilon(x)dx=\int_K \prod_{i=1}^n(\epsilon\lambda_i+1)dx.
\end{equation}
We aim to estimate the product $\prod_{i=1}^n(\epsilon\lambda_i+1)$ from below.
We have by Theorem~\ref{th:2.2}, that
\begin{equation}
\label{3.6}
\left(\prod_{i=1}^n\frac{\epsilon\lambda_i}{1+\epsilon\lambda_i}\right)^{1/n}\leq \frac{1}{n}\sum_{i=1}^n\frac{\epsilon\lambda_i}{1+\epsilon\lambda_i}-
\frac{1}{n}\sum_{i=1}^n\left(\sqrt{\frac{\epsilon\lambda_i}{1+\epsilon\lambda_i}}-u\right)^2,
\end{equation}
where
$$u=\left(\prod_{i=1}^n\frac{\epsilon\lambda_i}{1+\epsilon\lambda_i}\right)^{1/2n}.$$
Again, by the geometric-arithmetic mean inequality we get
\begin{equation}
\label{3.7}
\left(\prod_{i=1}^n\frac{1}{1+\epsilon\lambda_i}\right)^{1/n}\leq \frac{1}{n}\sum_{i=1}^n\frac{1}{1+\epsilon\lambda_i}.
\end{equation}
thus summing inequalities (\ref{3.6}) and (\ref{3.7}) we obtain
\begin{equation}
\label{3.8}
\frac{1+\epsilon\mu}{(\prod_{i=1}^n(\epsilon\lambda_i+1))^{1/n}}\leq 1-\frac{1}{n}\sum_{i=1}^n\left(\sqrt{\frac{\epsilon\lambda_i}{1+\epsilon\lambda_i}}-u\right)^2\leq 1.
\end{equation}
Next denote $v=\sqrt{\frac{\epsilon\mu}{1+\epsilon\mu}}.$ First of all the estimate (\ref{3.8}) implies, that $v\geq u.$ We aim to prove now, that
for $\epsilon>0$ small enough one has the estimate
\begin{equation}
\label{3.9}
\sum_{i=1}^n\left(\sqrt{\frac{\epsilon\lambda_i}{1+\epsilon\lambda_i}}-u\right)^2\geq \sum_{i=1}^n\left(\sqrt{\frac{\epsilon\lambda_i}{1+\epsilon\lambda_i}}-v\right)^2.
\end{equation}
Opening the brackets inequality (\ref{3.9}) amounts to the following
$$2(v-u)\sum_{i=1}^n\sqrt{\frac{\epsilon\lambda_i}{1+\epsilon\lambda_i}}\geq n(v^2-u^2),$$
thus taking into account the estimate $v\geq u,$ we get an equivalent inequality
\begin{equation}
\label{3.10}
2\sum_{i=1}^n\sqrt{\frac{\epsilon\lambda_i}{1+\epsilon\lambda_i}}\geq nu+nv.
\end{equation}
By the definition of $u$ and the geometric-arithmetic mean inequality we have
$$
\sum_{i=1}^n\sqrt{\frac{\epsilon\lambda_i}{1+\epsilon\lambda_i}}\geq nu,
$$
for all $\epsilon>0,$ thus it remains to show, that
\begin{equation}
\label{3.11}
\sum_{i=1}^n\sqrt{\frac{\epsilon\lambda_i}{1+\epsilon\lambda_i}}\geq nv,
\end{equation}
for small enough $\epsilon>0.$ By the smoothness of the mapping $F(x)$ and the positivity of the eigenvalues $\lambda_i(x),$ one has the following uniform estimates
\begin{equation}
\label{3.12}
0<a\leq \lambda_i(x)\leq b<\infty,\quad\text{uniformly in}\quad x\in K, \ \ i=1,2,\dots,n,
\end{equation}
thus $\epsilon\lambda_i(x)\leq \epsilon b\leq\frac{1}{2},$ for all $x\in\overline{K}$ and $i=1,2,\dots,n$ as long as $\epsilon\leq\frac{1}{2b}.$
This shows the validity of (\ref{3.11}) and thus (\ref{3.9}) owing to Lemma~\ref{lem:2.7}. Putting together now (\ref{3.8}) and (\ref{3.9}) we get the estimate
\begin{equation}
\label{3.13}
\frac{1+\epsilon\mu}{(\prod_{i=1}^n(\epsilon\lambda_i+1))^{1/n}}\leq 1-\frac{1}{n}\sum_{i=1}^n\left(\sqrt{\frac{\epsilon\lambda_i}{1+\epsilon\lambda_i}}-\sqrt{\frac{\epsilon\mu}{1+\epsilon\mu}}\right)^2.
\end{equation}
Next we estimate
\begin{align*}
\left(\sqrt{\frac{\epsilon\lambda_i}{1+\epsilon\lambda_i}}-\sqrt{\frac{\epsilon\mu}{1+\epsilon\mu}}\right)^2&=
\frac{\left(\frac{\epsilon\lambda_i}{1+\epsilon\lambda_i}-\frac{\epsilon\mu}{1+\epsilon\mu}\right)^2}
{\left(\sqrt{\frac{\epsilon\lambda_i}{1+\epsilon\lambda_i}}+\sqrt{\frac{\epsilon\mu}{1+\epsilon\mu}}\right)^2}\\
&\geq \frac{\left(\frac{\epsilon\lambda_i}{1+\epsilon\lambda_i}-\frac{\epsilon\mu}{1+\epsilon\mu}\right)^2}
{2\left(\frac{\epsilon\lambda_i}{1+\epsilon\lambda_i}+\frac{\epsilon\mu}{1+\epsilon\mu}\right)}\\
&=\frac{\epsilon^2(\lambda_i-\mu)^2}{2(1+\epsilon\lambda_i)(1+\epsilon\mu)(\epsilon\lambda_i+\epsilon\mu+\epsilon^2\lambda_i\mu)}\\
&\geq\frac{\epsilon(\lambda_i-\mu)^2}{2.1(\lambda_i+\mu)},
\end{align*}
provided $\epsilon$ is small enough. Therefore we get from (\ref{3.13}) the simpler looking estimate
$$
\frac{1+\epsilon\mu}{(\prod_{i=1}^n(\epsilon\lambda_i+1))^{1/n}}\leq 1-\frac{1}{2.1n}\sum_{i=1}^n\frac{\epsilon(\lambda_i-\mu)^2}{\lambda_i+\mu},
$$
which finally implies by the Bernoulli inequality
\begin{align}
\label{3.14}
\prod_{i=1}^n(\epsilon\lambda_i+1)&\geq (1+\epsilon\mu)^n\left(\frac{1}{1-\frac{1}{2.1n}\sum_{i=1}^n\frac{\epsilon(\lambda_i-\mu)^2}{\lambda_i+\mu}}\right)^n\\ \nonumber
&\geq (1+\epsilon\mu)^n\left(1+\frac{1}{2.1n}\sum_{i=1}^n\frac{\epsilon(\lambda_i-\mu)^2}{\lambda_i+\mu}\right)^n\\ \nonumber
&\geq (1+\epsilon\mu)^n\left(1+\frac{1}{2.1}\sum_{i=1}^n\frac{\epsilon(\lambda_i-\mu)^2}{\lambda_i+\mu}\right).
\end{align}
Denote for simplicity $U=\sum_{i=1}^n\frac{\epsilon(\lambda_i-\mu)^2}{\lambda_i+\mu}.$ Then we have combining the estimates (\ref{3.5})
and (\ref{3.14}) that
\begin{equation*}
|K+L_\epsilon|\geq (1+\epsilon\mu)^n\left(|K|+\frac{1}{2.1}\int_K Udx\right),
\end{equation*}
which gives
\begin{equation}
\label{3.15}
|K+L_\epsilon|^{1/n}\geq (|K|^{1/n}+|L_\epsilon|^{1/n})\left(1+\frac{1}{2.1|K|}\int_K Udx\right)^{1/n}.
\end{equation}
Next, owing to the bounds (\ref{3.12}) we can estimate
$$U=\sum_{i=1}^n\frac{\epsilon(\lambda_i-\mu)^2}{\lambda_i+\mu}\leq \frac{2nb^2}{a}\epsilon,$$
thus we have
$$\frac{1}{2.1|K|}\int_K Udx\leq \frac{nb^2}{a}\epsilon\to 0\quad\text{as}\quad\epsilon\to0,$$
and hence we have for small enough $\epsilon$ by the binomial expansion,
$$\left(1+\frac{1}{2.1|K|}\int_K Udx\right)^{1/n}\geq 1+\frac{1}{2.2n|K|}\int_K Udx,$$
which gives together with (\ref{3.15}) the estimate
\begin{equation*}
|K+L_\epsilon|^{1/n}\geq (|K|^{1/n}+|L_\epsilon|^{1/n})\left(1+\frac{1}{2.2n|K|}\int_K Udx\right),
\end{equation*}
which amounts to
\begin{equation}
\label{3.16}
\beta(K,L_\epsilon)\geq \frac{1}{2.2n|K|}\int_K Udx.
\end{equation}
In the next step we recall the following inequality proven by Figalli, Maggi and Pratelli in [\ref{bib:Fig.Mag.Pra.2}],
\begin{equation}
\label{3.17}
A(K,L_\epsilon)\leq \frac{C_0nq_K}{\epsilon\mu|K|}\int_{K}|\nabla F_\epsilon(x)-\epsilon\mu I|dx,
\end{equation}
where $C_0=\frac{2\sqrt 2}{\ln 2}$. For convenience of the reader we present the proof of (\ref{3.17}) from [\ref{bib:Fig.Mag.Pra.2}].
The key estimate needed for proving (\ref{3.17}) is the following trace inequality proven again in [\ref{bib:Fig.Mag.Pra.2}].
\begin{lemma}
\label{lm:3.1}
Let $K\subset\mathbb R^n$ be a convex body such that $B_r\subset K\subset B_R$ for some $0<r<R.$ Then
\begin{equation}
\label{3.18}
\frac{C_0nR}{2r}\int_K|\nabla f(x)|dx\geq \inf_{c\in\mathbb R}\int_{\partial K}|f(x)-c|d\mathcal{H}^{n-1},
\end{equation}
for all $f\in C^\infty(\mathbb R^n)\cap L^\infty(\mathbb R^n).$
\end{lemma}
The idea of the proof is that inequality (\ref{3.16}) insures, that the sets $\frac{1}{\epsilon\mu }L_\epsilon=\frac{1}{\epsilon\mu }F_\epsilon(K)$ and $K$ are close provided $\beta(K,L)$ is small. The strategy of estimating the measure of the symmetric difference $\left(\frac{1}{\epsilon\mu }L_\epsilon\right)\triangle K$ is the following: given a point $x$ on the boundary of the set $K$, one projects it onto the set $L'=\frac{1}{\mu }L$ and integrates the obtained distance over the boundary of $L'.$ Namely, Figalli, Maggi and Pratelli do the following calculation: Denote by $P(x)\colon\mathbb R^n\setminus L'\to\partial L'$ the projection onto the set $L',$ then as $F(x)$ takes values in $L,$ one gets the estimate
\begin{equation}
\label{3.19}
\frac{1}{|K|}\int_{\partial K}\left|\frac{F(x)}{\mu}- x\right|d\mathcal{H}^{n-1}\geq \frac{1}{|K|}\int_{\partial K\L'}|P(x)-x|d\mathcal{H}^{n-1}
\end{equation}
Consider now the map $\Phi(x,t)\colon(\partial K\setminus L')\times[0,1]\to K\setminus L'$ defined by
$$\Phi(x,t)=tx+(1-t)P(x).$$
It is clear, that as $\Phi(x,t)$ lies on the segment joining the points $x$ and $P(x),$ then $\Phi(x,t)$ is a bijection. Let now
$\{\epsilon_k(x)\}_{k=1}^{n-1}$ be a basis of the tangent space to $\partial K$ at $x.$ Since $\Phi$ is a bijection on has
$$|K\setminus L'|=\int_0^1dt\int_{\partial K\setminus L'}\left|(x-P(x))\wedge \left(\bigwedge_{k=1}^n(t\epsilon_k(x)+(1-t)dP_x(\epsilon_k(x)))\right)\right|d\mathcal{H}^{n-1}.$$
Since $P(x)$ is a projection onto a convex set, it decreases the distances [\ref{bib:Sch.}], thus $|dP_x|\leq 1,$
thus we have $|t\epsilon_k(x)+(1-t)dP_x(\epsilon_k(x))|\leq 1$ for all $k$ and $x.$
Therefore one gets from the last equality, that
$$\frac{|K\setminus L'|}{|K|}\leq \frac{1}{|K|}\int_{\partial K\setminus L'}|P(x)-x|d\mathcal{H}^{n-1},$$
which together with (\ref{3.19}) implies
\begin{equation}
\label{3.20}
\frac{1}{|K|}\int_{\partial K}\left|\frac{F(x)}{\mu}- x\right|d\mathcal{H}^{n-1}\geq \frac{|K\setminus L'|}{|K|}.
\end{equation}
One can have assumed initially, that the set $L'$ is translated by a vector $c\in\mathbb R^n$ so that
$$\int_{\partial K}\left|\frac{F(x)}{\mu}- x\right|d\mathcal{H}^{n-1}=\inf_{c\in\mathbb R^n}\int_{\partial K}\left|\frac{F(x)}{\mu}- x-c\right|d\mathcal{H}^{n-1}.$$

Thus finally noticing, that $A(K,L)\leq \frac{|K\triangle L'|}{|K|}=2\frac{|K\setminus L'|}{|K|}$ and applying Lemma~\ref{lm:3.1},
the estimate (\ref{3.17}) follows from (\ref{3.20}). The rest of the analysis is to derive the estimate (\ref{3.0.1}) from (\ref{3.16}) and
 (\ref{3.17}). To that end we denote
\begin{equation}
\label{3.21}
V=|\nabla F-\mu I|=\left(\sum_{i=1}^n(\lambda_i-\mu)^2\right)^{1/2},\quad W=\sum_{i=1}^n(\lambda_i+\mu).
\end{equation}
We have on one hand by the Schwartz inequality, that
\begin{align*}
V&=\left(\sum_{i=1}^n(\lambda_i-\mu)^2\right)^{1/2}\\ \nonumber
&\geq\frac{1}{\sqrt{n}}\sum_{i=1}^n|\lambda_i-\mu|\\ \nonumber
&\geq\frac{1}{\sqrt{n}}\sum_{i=1}^n(\lambda_i-\mu)\\ \nonumber
&=\frac{1}{\sqrt{n}}(W-2n\mu),
 \end{align*}
 thus we get
\begin{equation}
\label{3.22}
W\leq \sqrt{n}V+2n\mu.
\end{equation}
We have on the other hand again by the Schwartz inequality and utilizing (\ref{3.16}), that
\begin{align*}
2.2n|K|\beta(K,L_\epsilon)\int_K(\sqrt{n}V+2n\mu)dx&\geq \int_{K}Udx\int_KWdx\\
&\geq \left(\int_{K}\sqrt{UW}dx\right)^2\\
&\geq \epsilon\left(\int_{K}\sum_{i=1}^n|\lambda_i-\mu|dx\right)^2\\
&\geq \epsilon\left(\int_{K}Vdx\right)^2,
\end{align*}
which then gives the estimate
\begin{equation}
\label{3.23}
\frac{1}{|K|}\int_K Vdx\leq \frac{2.2n\sqrt n\beta(K,L_\epsilon)}{\epsilon}+\sqrt{\frac{4.2n^2\mu \beta(K,L_\epsilon)}{\epsilon}}.
\end{equation}
Recall, that we are after the estimate (\ref{3.0.1}), which is equivalent to
\begin{equation}
\label{3.24}
A(K,L_\epsilon)^2\leq \frac{100n^4q_K^2\beta(K,L_\epsilon)}{\mu\epsilon}.
\end{equation}
It is clear, that $A(K,L_\epsilon)\leq 2,$ thus we can without loss of generality assume, that
$$\frac{100n^4q_K^2\beta(K,L_\epsilon)}{\mu\epsilon}<4,$$
thus taking into account the bound $q_K\geq 1$ and $n\geq 2,$ we get the estimate
\begin{equation}
\label{3.25}
\frac{\beta(K,L_\epsilon)}{\mu\epsilon}<\frac{1}{25n^4}\leq\frac{1}{200n}.
\end{equation}
It is then easy to see, that (\ref{3.25}) implies, that
$$\frac{2.1n\sqrt n\beta(K,L_\epsilon)}{\epsilon}\leq \frac{1}{10}\sqrt{\frac{4.2n^2\mu \beta(K,L_\epsilon)}{\epsilon}},$$
and hence we discover from (\ref{3.23})
\begin{equation}
\label{3.26}
\frac{1}{|K|}\int_K Vdx\leq \frac{11}{10}\sqrt{\frac{4.2n^2\mu \beta(K,L_\epsilon)}{\epsilon}}.
\end{equation}
Finally, combining now the estimates (\ref{3.17}) and (\ref{3.24}) we arrive at
\begin{equation}
\label{3.27}
A^2(K,L_\epsilon)\leq \frac{121\cdot 4.2C_0^2n^4q_K^2}{100\epsilon\mu}\beta(K,L_\epsilon),
\end{equation}
which yields (\ref{3.0.1}). The proof of the lemma is finished now.
\end{proof}

It is  a well known procedure how the proof of Theorem~\ref{th:1.1} easily follows now from the estimate (\ref{3.0.1}), one just lets $\epsilon$ go to zero in (\ref{3.0.1}). Indeed, from the estimate
$$|K+L_\epsilon|\geq \left(|K|^{1/n}+\epsilon |L|^{1/n}\right)^n\left(1+\frac{\epsilon A(K,L)^2|L|^{1/n}}{C(K,n)|K|^{1/n}}\right),$$
we get by the well known (Minkowski-Steiner formula for the case of $L=B_1$) formula [\ref{bib:Fed.}], that at the first order (as $\epsilon\to 0$),
\begin{align*}
|K|+\epsilon P_L(K)&\geq (|K|+\epsilon n|K|^{(n-1)/n}|L|^{1/n})\left(1+\frac{\epsilon A(K)^2|L|^{1/n}}{C(K,n)|K|^{1/n}}\right)^n\\
&\geq (|K|+\epsilon n|K|^{(n-1)/n}|L|^{1/n})\left(1+\frac{\epsilon nA(K)^2|L|^{1/n}}{C(K,n)|K|^{1/n}}\right),
\end{align*}
which is exactly (\ref{1.11}). The proof of the theorem is finished now.
\end{proof}

\begin{proof}[Proof of Theorem~\ref{th:1.3}] The derivation of a Brunn-Minkowski for convex sets from an anisotropic isoperimetric inequality is again classical and is due to Hadwiger and Ohmann [\ref{bib:Had.Ohm.}]. Again, for convenience of the reader we present the proof here.
From the definition of the anisotropic perimeter, it is clear that
\begin{equation}
\label{3.29}
P_K(M)+P_L(M)=P_{K+L}(M),
\end{equation}
for all convex sets $K,L,M\in\mathbb R^n.$ Another trivial and classical fact is the triangle inequality
\begin{equation}
\label{3.30}
A(K,L)\leq A(K,M)+A(M,L).
\end{equation}
We can assume without loss of generality, that the origin is an inner point for both of the sets $K$ and $L$ and it is also clear that the set $M=K+L$ is convex too. Thus we have by Corollary~\ref{cor:1.2}, that
\begin{equation*}
P_K(K+L)\geq n|K+L|^{(n-1)/n}|K|^{1/n}\left(1+\frac{A(K+L,K)^2}{C(n)}\right),
\end{equation*}
and
\begin{equation*}
P_L(K+L)\geq n|K+L|^{(n-1)/n}|L|^{1/n}\left(1+\frac{A(K+L,L)^2}{C(n)}\right),
\end{equation*}
thus summing the two estimates and owing to (\ref{3.29}) we get
\begin{align}
\label{3.31}
\frac{|K+L|^{1/n}}{|K|^{1/n}+|L|^{1/n}}-1&\geq \frac{1}{C(n)}\left(\frac{|K|^{1/n}}{|K|^{1/n}+|L|^{1/n}}A(K+L,K)^2+\frac{|L|^{1/n}}{|K|^{1/n}+|L|^{1/n}}A(K+L,L)^2\right)\\ \nonumber
&\geq \frac{1}{C(n)}\left(\frac{\sigma^{1/n}}{1+\sigma^{1/n}}A(K+L,K)^2+\frac{1}{1+\sigma^{1/n}}A(K+L,L)^2\right)\\ \nonumber
&\geq \frac{1}{C(n)}\left(\frac{1}{2\sigma^{1/n}}A(K+L,K)^2+\frac{1}{2\sigma^{1/n}}A(K+L,L)^2\right),
\end{align}
as $\sigma\geq 1.$ An application of the inequality $a^2+b^2\geq \frac{1}{2}(a+b)^2$ and due to the triangle inequality we get
\begin{equation*}
\label{3.32}
\frac{|K+L|^{1/n}}{|K|^{1/n}+|L|^{1/n}}-1\geq \frac{A(K,L)^2}{4C(n)\sigma^{1/n}},
\end{equation*}
which completes the proof of the theorem.
\end{proof}
Finally, we comment on the Conjecture~\ref{con:1.4}. As we have already seen, Brunn-Minkowski inequality implies the anisotropic
isoperimetric inequality with the same constant $C(n),$ and the anisotropic isoperimetric inequality with a constant $C(n)$ implies the Brunn-Minkowski inequality with a constant $4C(n),$ thus it suffices to prove that $C(n)\geq Cn^2$ for some constant $C$ in the Brunn-Minkowski
inequality. To that end we consider the two boxes $K=[0,1]^n$ and $L=[0,1]^m\times[0,1+\epsilon]^{n-m},$ where $n\geq 2,$ $m=\left[n/2\right]$ is the whole part of $n/2,$ and $\epsilon$ is a small number. We have denoting $\alpha=(n-m)/n,$ by Taylor's formula that
\begin{align}
\label{3.33}
\beta(K,L)&=\frac{|K+L|^{1/n}-|K|^{1/n}-|L|^{1/n}}{|K|^{1/n}+|L|^{1/n}}\\ \nonumber
&\leq \frac{1}{2}(|K+L|^{1/n}-|K|^{1/n}-|L|^{1/n})\\ \nonumber
&=\frac{1}{2}\left(2\left(1+\frac{\epsilon}{2}\right)^{\alpha}-1-(1+\epsilon)^{\alpha}\right)\\ \nonumber
&=\frac{1}{2}\left(2\left(1+\frac{\alpha\epsilon}{2}+\frac{\alpha(\alpha-1)\epsilon^2}{8}+O(\epsilon^3)\right)-1-\left(1+\alpha\epsilon+
\frac{\alpha(\alpha-1)\epsilon^2}{2}+O(\epsilon^3)\right)\right)\\ \nonumber
&=\frac{\alpha(1-\alpha)\epsilon^2}{4}+O(\epsilon^3)\\ \nonumber
&\leq \frac{\epsilon^2}{16}+O(\epsilon^3).
\end{align}
On the other hand by the construction of the boxes $K$ and $L$ it is clear, that
\begin{equation}\label{3.34}
A(K,L)\geq cn\epsilon ,
\end{equation}
for some constant $c>0.$ Thus combining inequalities (\ref{3.33}), (\ref{3.34}) and (\ref{1.6}), and sending $\epsilon$ to zero we arrive at the estimate $C(n)\geq Cn^2$ for some absolute constant $C>0.$ This insures, that the exponent $2$ of $n$ in the constant $C(n)$ is not possible to
make any lower. The reverse inequality for the optimal constant $C(n)\leq Cn^2$ is a task for future. It is also worth mentioning, that Segal showed in [\ref{bib:Seg.}], that if one assumes that the validity of Dar's conjecture implies the estimate $C(n)\leq Cn^2$, i.e., proves Conjecture~\ref{con:1.4}. Recall, that Dar's conjecture [\ref{bib:Dar}] asserts the following: \textit{For any bounded convex bodies, the inequality holds}
$$|K+L|^{1/n}\geq M(K,L)^{1/n}+\frac{|K|^{1/n}|L|^{1/n}}{M(K,L)^{1/n}},$$
where
$$M(K,L)=\max_{x\in\mathbb R^n}|K\cap(L+x)|.$$
Note, that Dar's conjecture has been proven recently by Xi and Leng [\ref{bib:Xi.Len.}] in the planar case $n=2.$
\begin{remark}
\label{rem:3.4}
As mentioned in the introduction section, in the general case when the sets $X,Y\subset\mathbb R^n$ are just bounded and measurable, then the validity of (\ref{1.6}) is open. The convex body proof approach does not work in the general case due to many facts in particular the luck of John's symmetrization and the passage from Wullf's inequality to the Brunn-Minkowski.
\end{remark}

\section{Acknowledgement}

The present results have been obtained while the author was a postdoctoral fellow at the University of Utah. The author is very grateful to
Graeme W. Milton for supporting his stay at University of Utah. The author is also grateful for the anonymous referees to point out papers [10] and [34] and also for valuable comments that improved the presentation and the text of the manuscript.


\begin{thebibliography}{999}

\bibitem{} \label{bib:Bal.} Ball K., An elementary introduction to monotone transportation. \textit{Geometric Aspects of Functional Analysis} (Israel Seminar
2002-2003), Lecture Notes in Math. 1850, pp. 41-52, Springer, Berlin, 2004.

\bibitem{} \label{bib:Ben.} T. Bonnesen, \"Uber die isoperimetrische Defizit ebener Figuren, \textit{Math. Ann.} 91 ( 1924), 252--268.

\bibitem{} \label{bib:Ber.}  F. Bernstein, Uber die isoperimetrische Eigenschaft des Kreises auf der K\"ugeloberfl\"ache
und in der Ebene, \textit{Math. Ann.,} 60 (1905), 117-136.

\bibitem{} \label{bib:Bre.1} Y. Brenier, D\'ecomposition polaire et r\'earrangement monotone de champs de vecteurs,
 \textit{R. Acad. Sci. Paris S\'er. I Math.,} 305 no. 19 (1987), 805-808.

\bibitem{} \label{bib:Bre.2} Y. Brenier, Polar factorization and monotone rearrangement of vector-valued functions,
\textit{Comm. Pure Appl. Math.} 44 (4) (1991), 375-417.

\bibitem{} \label{bib:Bro.Mor.}  J. E. Brothers and F. Morgan, The isoperimetric theorem for general integrands, \textit{ Michigan
Math. J.} 41 (1994), no. 3, 419-431.

\bibitem{} \label{bib:Bur.Zal.} Y. D. Burago and V. A. Zalgaller, \textit{Geometric inequalities,} Springer, New York, 1988. Russian
original: 1980.

\bibitem{} \label{bib:Caf.1} L. A. Caffarelli, The regularity of mappings with a convex potential, \textit{J. Amer. Math. Soc.,} 5 (1992), no. 1, 99-104.

\bibitem{} \label{bib:Caf.2} L. A. Caffarelli, Boundary regularity of maps with convex potentials, II., \textit{Ann. of Math.} (2) 144 (1996), no. 3. 453-496.

\bibitem{} \label{bib:Car.Mag.} E.A. Carlen and F. Maggi, Stability for the Brunn-Minkowski and Riesz rearrangement inequalities, with applications to Gaussian concentration and finite range non-local isoperimetry. \textit{Preprint,} https://arxiv.org/abs/1507.03454

\bibitem{} \label{bib:Dac.Pfi.} B. Dacorogna and C. E. Pfister, Wulff theorem and best constant in Sobolev inequality, \textit{J.
Math. Pures Appl.} (9) 71 (2) (1992) 97-118.

\bibitem{} \label{bib:Dar} S. Dar, A Brunn-Minkowski-Type inequality, \textit{Geom. Dedicata} 77 (1999), 1-9,
MR 1706512, Zbl 0938.52008.

\bibitem{} \label{bib:Esp.Fus.Tro.} L. Esposito, N. Fusco and C. Trombetti, A quantitative version of the isoperimetric inequality:
the anisotropic case. \textit{Ann. Sc. Norm. Super. Pisa Cl. Sci.} (5) 4 (2005), no. 4, 619–651.


\bibitem{} \label{bib:Fig.Jer.} A. Figalli and D. Jerison, Quantitative stability for the Brunn-Minkowski inequality, \textit{Preprint}, https://arxiv.org/abs/1502.06513

\bibitem{} \label{bib:Fig.Mag.Pra.1} Figalli, A.; Maggi, F.; Pratelli, A. A mass transportation approach to quantitative isoperimetric
inequalities. \textit{Invent. Math.} 182 (2010), no. 1, 167--211.

\bibitem{} \label{bib:Fig.Mag.Pra.2} Figalli, A.; Maggi, F.; Pratelli, A. A refined Brunn-Minkowski inequality for convex sets.
 \textit{Ann. Inst. H. Poincar\'e Anal. Non Lin\'eaire} 26 (2009), no. 6, 2511--2519.

\bibitem{} \label{bib:Fed.} H. Federer,  \textit{ Geometric measure theory.} Die Grundlehren der mathematischen Wissenschaften,
Band 153 Springer-Verlag New York Inc., New York, 1969 xiv+676 pp.

\bibitem{} \label{bib:Fon.Mue.}  I. Fonseca and S. M\"uller, A uniqueness proof for the Wulff theorem. \textit{Proc. Roy. Soc. Edinburgh
Sect. A} 119 (1991), no. 1-2, 125-136.

\bibitem{} \label{bib:Fug.} B. Fuglede, Stability in the isoperimetric problem for convex or nearly spherical domains
in $\mathbb R^n$, \textit{Trans. Amer. Math. Soc., } 314 (1989), 619-638.

 \bibitem{} \label{bib:Fus.Mag.Pra.} N. Fusco, F. Maggi and A. Pratelli, The sharp quantitative isoperimetric inequality, \textit{Ann.
of Math.} 168 (2008), 941-980.

\bibitem{} \label{bib:Gar.} R. J. Gardner, The Brunn-Minkowski inequality, \textit{Bull. Amer. Math. Soc. (N.S.)} 39 (2002),
no. 3, 355-405.

\bibitem{} \label{bib:Gro.1} H. Groemer, \textit{Stability of geometric inequalities, Handbook of Convexity,} ed. by P. M. Gruber and J. M. Wills,
North--Holland, Amsterdam, 1993, pp. 125--150.

\bibitem{} \label{bib:Gro.2} Groemer, H, On the Brunn-Minkowski theorem. \textit{Geom. Dedicata,} 27 (1988), no. 3, 357-371.

\bibitem{} \label{bib:Had.Ohm.}  H. Hadwiger and D. Ohmann, Brunn-Minkowskischer Satz und Isoperimetrie, \textit{Math. Zeit.} 66 (1956), 1-8.

\bibitem{} \label{bib:Hal.} R. R. Hall, A quantitative isoperimetric inequality in n-dimensional space, \textit{J. Reine
Angew. Math,} 428 (1992), 161-176.

\bibitem{} \label{bib:Har.Lit.Pol.} G. H. Hardy, J. E. Littlewood, and G. P\'olya, \textit{Inequalities,} Cambridge University Press, Cambridge, 1959.

\bibitem{} \label{bib:Hen.Mac.} R. Henstock and A. M. Macbeath, On the measure of sum sets, I. The theorems of Brunn,
Minkowski and Lusternik, \textit{ Proc. London Math. Soc.} 3 (1953), 182--194.

\bibitem{} \label{bib:Hor.} L. H\"ormander, \textit{The Analysis of Linear Partial Differential Operators I: Distribution Theory and Fourier Analysis,} (Classics in Mathematics), Springer, 2nd ed. 1990

\bibitem{} \label{bib:Joh.} F. John, An inequality for convex bodies, \textit{ Univ. Kentucky Research Club Bull.,} 8 (1942), 8-11.

\bibitem{} \label{bib:Kno.} H. Knothe, Contributions to the theory of convex bodies, \textit{ Michigan Math. J.} 4 (1957) 39-52.

\bibitem{} \label{bib:McC.1.} R.J. McCann, Existence and uniqueness of monotone measure-preserving maps, \textit{Duke
Math. J.} 80 (2) (1995) 309-323.

\bibitem{} \label{bib:McC2.}  R.J. McCann, A convexity principle for interacting gases, \textit{Adv. Math.} 128 (1) (1997) 153-179.

\bibitem{} \label{bib:Mil.Sch.} V. D. Milman and G. Schechtman, \textit{Asymptotic theory of finite-dimensional normed spaces.
With an appendix by M. Gromov.} Lecture Notes in Mathematics, 1200. Springer-Verlag,
Berlin, 1986. viii+156 pp.

\bibitem{} \label{bib:Sch.} J. Van Schaftingen, Anisotropic symmetrization, \textit{Ann. Inst. H. Poincar´e Anal. Non Lineaire}
23 (2006), no. 4, 539-565.

\bibitem{} \label{bib:Seg.} A. Segal, Remark on stability of Brunn-Minkowski and isoperimetric inequalities for convex
bodies. \textit{Geometric aspects of functional analysis,} 381-391, Lecture Notes in Mathematics, 2050, Springer, Heidelberg, 2012.

\bibitem{} \label{bib:Sch.} Schneider, R. \textit{Convex bodies: The Brunn-Minkowski theory.} Encyclopedia of Mathematics and
its Applications, 44. Cambridge University Press, Cambridge, 1993.

\bibitem{} \label{bib:Vil.} C. Villani, \textit{Topics in optimal transportation, Graduate Studies in Mathematics,} 58. American
Mathematical Society, Providence, RI, 2003. xvi+370 pp.

\bibitem{} \label{bib:Xi.Len.} Dongmeng Xi and Gangsong Leng, Dar's conjecture and the log-Brunn-Minkowski inequality. \textit{Journal of Differential Geometry,} 103 (2016), 145-189.

\bibitem{} \label{bib:Wul.} G. Wulff. Zur Frage der Geschwindigkeit des Wachsturms und der Aufl\"osung der
Kristallfl\"achen, \textit{Z. Kristallogr.} 34, 449-530.

\end{thebibliography}
\end{document}